\documentclass{article}
\usepackage{amsmath,amsthm,amscd,amssymb}
\usepackage[mathscr]{eucal}
\usepackage{amssymb}
\usepackage{latexsym}

\theoremstyle{definition}
\newtheorem{Def}{Definition}[section]
\newtheorem{Thm}[Def]{Theorem}

\numberwithin{equation}{section}

\title{Notes on theta series for Niemeier lattices II}

\author{Shoyu Nagaoka}

\date{}

\begin{document}

\maketitle

\begin{abstract}
Following \cite{N-T},
we study some congruence properties satisfied by the theta series associated with 
Niemeier lattices. 
\end{abstract}
%%%%%%%%%%INTRODUCTION%%%%%%%%%%%%
\section{Introduction}
This is a continuation of \cite{N-T}, where some congruence relations
that a Siegel theta series $\vartheta_{\mathcal{L}}^{(n)}$ satisfies for a Niemeier
lattice $\mathcal{L}$ were studied.
The beginning of this study was the discovery of the congruence relation
$$
\varTheta (\vartheta_\omega^{(2)}) \equiv 0 \pmod{23},
$$
where $\varTheta$ is the theta operator and $\omega$ is the Leech lattice
(for the definition, see \cite[$\S 2$]{N-T}). In \cite{N-T}, it was shown
that this relation is a consequence of the fact that $\vartheta_\omega^{(2)}$ is
congruent to the theta series for binary quadratic forms with discriminant
$-23$ (\cite[Theorem 3]{N-T}).
It should be noted that the prime number $23$ is a prime factor
of $|\text{Aut}(\omega)|$ (the order of the automorphism group of the
Leech lattice $\omega$).

In this study, we present another type of congruence relations, namely,
{\it mod 11 congruences} between Siegel theta series for Niemeier lattices
and theta series for quaternary quadratic forms.

Let $\{S_1,S_2,S_3 \}$ be a set of representatives of the unimodular
equivalence classes of the genus of the quaternary quadratic forms of
discriminant $11^2$ with level $11$ (for the explicit form, see $\S$ \ref{sec3-1}).
\\
The first result is as follows:
\\
\begin{Thm}
{\it
The following congruence relations hold:
\begin{align*}
& \vartheta_\alpha^{(3)} \equiv \vartheta_\kappa^{(3)}
                                    \equiv \vartheta_\psi^{(3)} \equiv \vartheta_{S_1}^{(3)} \pmod{11},
                                    \\
& \vartheta_\delta^{(3)} \equiv \vartheta_\iota^{(3)}
                                    \equiv \vartheta_\chi^{(3)} \equiv \vartheta_{S_2}^{(3)} \pmod{11},
                                    \\
& \vartheta_\epsilon^{(3)} \equiv \vartheta_\omega^{(3)}
                                    \equiv \vartheta_{S_3}^{(3)} \pmod{11}.
\end{align*}
where the Greek characters $\alpha,\,\kappa,\ldots$ denote the lattices labeled
as in {\rm \cite[Table 16.1]{C-S}}. }
\end{Thm}
%%%%%%
In \cite{O}, Ozeki obtained a large number of numerical data regarding the Fourier coefficients of Siegel
theta series $\vartheta_\omega^{(n)}$ for the Leech lattice $\omega$. His calculations
allow the generalization of a congruence relation that $\vartheta_\omega^{(4)}$
satisfies.
%%%
\begin{Thm}
{\it We obtain
$$
\vartheta_\omega^{(4)} \equiv \vartheta_{S_3}^{(4)} \pmod{11}.
$$
In particular, the congruence relation
$$
\varTheta (\vartheta_\omega^{(4)}) \equiv 0 \pmod{11}
$$
holds. Namely, $\vartheta_\omega^{(4)}$ is in the mod 11 kernel of the theta
operator $\varTheta$.}
\end{Thm}
%%%%%%%%%%
In the final section, we study some congruence relations that
the theta series $\vartheta_\omega^{(n)}$ satisfies.

%%%%%%%%%%%%%%%%%%%%%%%%%%%%%%%%%%%%%%%%%%%%%%%%%%%%%%%%%%%%%%%%%%%%%%%%
\section{Preliminaries}
\label{sec:2}
%%%%%%%%%%
\subsection{Notation}
\label{sec:2-1}
The present notation primarily follows from \cite{N-T}. 

Let $\Gamma_n:=Sp_n(\mathbb{Z})$ be the Siegel modular group of degree $n$,
and let $\mathbb{H}_n$ be the Siegel upper half-space of degree $n$.
For a subgroup  $\Gamma\subset \Gamma_n$,
we denote by $M_k(\Gamma)$ the
$\mathbb{C}-$vector space of all Siegel modular forms of weight $k$ for $\Gamma$.

Any $F(Z)$ in  $M_k(\Gamma_n)$ has a Fourier expansion of the form
$$
F(Z)=\sum_{0\leq T\in Sym_n^*(\mathbb{Z})}a(F;T)q^T,\quad
         q^T:=\text{exp}(2\pi\sqrt{-1}\text{tr}(TZ)),\;Z\in\mathbb{H}_n,
$$
where
$$
Sym_n^*(\mathbb{Z}):=\{\,T=(t_{ij})\in Sym_n(\mathbb{Q})\,\mid\,
                                        t_{ii}, 2t_{ij}\in \mathbb{Z}\,\}.
$$
For a subring $R$ in $\mathbb{C}$, let $M_k(\Gamma_n)_R\subset M_k(\Gamma_n) $
denote the $R$-module of all modular forms whose Fourier coefficients lie in $R$.
As explained in \cite{N-T}, we may consider $F\in M_k(\Gamma_n)_R $
to be an element of the formal power series ring
$$
F\in \sum a(F;T)q^T\in R[q_{ij},q_{ij}^{-1}][\![q_{11},\ldots,q_{nn}]\!]
$$
when we write $q_{ij}:=\text{exp}(2\pi\sqrt{-1}z_{ij})$(cf. \cite[\S 2.2]{N-T}).

For a prime number $p$, we denote by $\mathbb{Z}_{(p)}$ the ring of $p$-integral
rational numbers.
For two elements
$$
F_m=\sum a(F_m;T)q^T\in\mathbb{Z}_{(p)}[q_{ij},q_{ij}^{-1}][\![q_{11},\ldots,q_{nn}]\!]\quad(m=1,2),
$$
we write $F_1 \equiv F_2 \pmod{p}$ if the congruence relation
$$
a(F_1;T) \equiv a(F_2;T) \pmod{p}
$$
is satisfied for all $T$ with $0\leq T\in Sym_n^*(\mathbb{Z})$.

 We associate with $F=\sum a(F;T)q^T\in M_k(\Gamma_n)$ the formal
power series
$$
\varTheta (F):=\sum a(F;T)\cdot\text{det}(T)q^T
\in\mathbb{C}[q_{ij},q_{ij}^{-1}][\![q_{11},\ldots,q_{nn}]\!].
$$
It should be noted that $\varTheta (F)$ is not necessarily a modular form.
The operator $\varTheta$ is called {\it the theta operator} (cf. \cite[\S 2.4]{N-T}).
%%%%%%%%%%%%%%%%%%%
%%%%%%%%%%%%
\subsection{Theta series for lattices and matrices}
\label{sec:2-2}
For a positive definite integral lattice $\mathcal{L}$ of degree $m$, we write the Gram
matrix as $S=S_{\mathcal{L}}\in Sym_m(\mathbb{Z})$. We associate with it the theta
series
$$
\vartheta_{\mathcal{L}}^{(n)}=\vartheta_S^{(n)}(Z):=
\sum_{X\in M_{m\times n}(\mathbb{Z})}\text{exp}(\pi\sqrt{-1}\text{tr}(S[X]Z),\;
Z\in \mathbb{H}_n,
$$
where $S[X]={}^tXSX$. In particular
$$
\vartheta_{\mathcal{L}}^{(n)}=\vartheta_S^{(n)}(Z)\in M_{\frac{m}{2}}(\Gamma_n)_{\mathbb{Z}}
$$
if $\mathcal{L}$ is an even unimodular lattice of degree $m$.

We now quote the following result, which will be used in the sequel and is a special case of
a theorem by B\"{o}cherer and Nagaoka.

\begin{Thm}
\label{Theorem2.1}
(B\"{o}cherer--Nagaoka \cite[Theorem 4]{B-N})\;{\it
We assume that $p\geq n+3$ and $p \equiv 3 \pmod{4}$. Let $S\in 2Sym_4^*(\mathbb{Z})$
be an even integral, positive definite symmetric matrix with $\text{det}S=p^2$ with
level $p$. Then, $\vartheta_S^{(n)}\in M_2(\Gamma_0^n(p))$, and there exists a
modular form $G\in M_{p+1}(\Gamma_n)_{\mathbb{Z}_{(p)}}$ such that
$$
\vartheta_S^{(n)} \equiv G \pmod{p}.
$$}
\end{Thm}
\begin{proof}
We apply \cite[Theorem 4]{B-N}, where
$$
f=\vartheta_S^{(n)}\in M_2(\Gamma_0^n(p))^0,\quad
g=G\in M_{p+1}(\Gamma_n)
$$
in the notation of \cite{B-N}.
\end{proof}
%%%%%%%%%%%%%%%%%%%
\subsection{Sturm bounds for Siegel modular forms}
\label{sec:2-3}
We introduce a result of Richter and Westerholt-Raum \cite{R-R}
concerning the so-called Sturm bound.
%%%%%%%%%
\begin{Thm}
\label{Theorem2.2}
(Richter--Westerholt-Raum \cite{R-R}) \;{\it 
We assume that $p$ is a prime number and $F=\sum a(F;T)q^T$ is a modular
form in $M_k(\Gamma_n)_{\mathbb{Z}_{(p)}}$ $(n\geq 2)$. If
$$
a(F;T) \equiv 0 \pmod{p}
$$
for all $0\leq T=(t_{ij})\in Sym_n^*(\mathbb{Z})$ with
$$
t_{ii}\leq \left(\frac{4}{3}\right)^n\frac{k}{16}\quad
(i=1,\ldots,n),
$$
then
$$
a(F;T) \equiv 0 \pmod{p}
$$
for all $0\leq T\in Sym_n^*(\mathbb{Z})$.}
\end{Thm}
%%%%%%%%%%%%%%%%%%%%%%%%%%%%%%%%%%%%%%%%%%%%%%%%%%%%%%%%%%%%%%%%%%%%%%%%
\section{Congruence properties of theta series for Niemeier lattices}
\label{sec:3}
%%%%%%%%%%
In \cite{N-T}, some congruence properties of theta series for Niemeier lattices
were obtained. In this section, we present more such properties.
\subsection{Quaternary quadratic forms of discriminant $11^2$}
\label{sec3-1}
For later use, we consider some quaternary quadratic forms.

It follows from Nipp's table \cite{N} that the genus of even positive quaternary
quadratic forms of discriminant $11^2$ with level 11 consists of three classes, and
their representatives are given by
\begin{equation}
\label{Si}
S_1=
\begin{pmatrix}
2 & 0 & 1 & 0 \\
0 & 2 & 0 & 1 \\
1 & 0 & 6 & 0 \\
0 & 1 & 0 & 6
\end{pmatrix},
\quad
S_2=
\begin{pmatrix}
2 & 1 & 1 & 1 \\
1 & 2 & 0 & 1 \\
1 & 0 & 8 & 4 \\
1 & 1 & 4 & 8
\end{pmatrix},
\quad
S_3=
\begin{pmatrix}
4 & 2 & 1 & 1 \\
2 & 4 & 0 & 1 \\
1 & 0 & 4 & 2 \\
1 & 1 & 2 & 4
\end{pmatrix}
\end{equation}
By definition, we see that
$$
\vartheta_{S_i}^{(n)}\in M_2(\Gamma_0^n(11))_{\mathbb{Z}}\quad (i=1,2,3).
$$
\subsection{Congruence properties of theta series for Niemeier lattices}
\label{sec3-2}
In this subsection we prove some congruence relations that theta series for Niemeier
lattices satisfy, which is our main result.
We introduce the symbols that denote Niemeier lattices before describing
the results. In \cite{C-S},  Conway and Sloane correspond the
Greek characters $\alpha,\ldots, \omega$ to Niemeier lattices
in the order of their Coxeter number.
 (cf. \cite[Table 16.1]{C-S}).
The last lattice $\omega$ is the Leech lattice.
%%%%%%%%%%
\begin{Thm}
\label{Theorem3.1}
{\it The following congruence relations hold:
\begin{align*}
{\rm (i)}\qquad&\vartheta_\alpha^{(3)} \equiv \vartheta_\kappa^{(3)}
                                    \equiv \vartheta_\psi^{(3)} \equiv \vartheta_{S_1}^{(3)} \pmod{11},
                                    \\
{\rm (ii)}\qquad&\vartheta_\delta^{(3)} \equiv \vartheta_\iota^{(3)}
                                    \equiv \vartheta_\chi^{(3)} \equiv \vartheta_{S_2}^{(3)} \pmod{11},
                                    \\
{\rm (iii)}\qquad&\vartheta_\epsilon^{(3)} \equiv \vartheta_\omega^{(3)}
                                    \equiv \vartheta_{S_3}^{(3)} \pmod{11}.
\end{align*}
where $\alpha, \kappa,\cdots$ are Niemeier lattices and $S_1,S_2,S_3$ are the Gram
matrices given in (\ref{Si}).}
\end{Thm}
%%%%%%
\begin{proof}
(i) Let $h_{\mathcal{L}}$ denote the Coxeter number of the Niemeier lattice $\mathcal{L}$.
We use the fact that the congruence relation 
$h_{\mathcal{L}_1} \equiv h_{\mathcal{L}_2} \pmod{n}$ implies 
$\vartheta_{\mathcal{L}_1}^{(3)} \equiv \vartheta_{\mathcal{L}_2}^{(3)} \pmod{n}$
(cf. \cite[Corollary 4]{N-T}).
As $h_\alpha=46$, $h_\kappa=13$, and $h_\psi=2$, we have
$h_\alpha \equiv h_\kappa \equiv h_\psi \equiv 2 \pmod{11}$. This implies
$$
\vartheta_\alpha^{(3)} \equiv \vartheta_\kappa^{(3)}
                                 \equiv \vartheta_\psi^{(3)} \pmod{11}.
$$
We will prove the congruence 
\begin{equation}
\label{maincong}
\vartheta_\alpha^{(3)} \equiv \vartheta_{S_1}^{(3)} \pmod{11}.
\end{equation}
To this end, we use the results for the Sturm bound. 
Specifically, we first replace $\vartheta_{S_1}^{(3)}$ with a modular form for $\Gamma_3$
in the sense of modulo 11. 

We note that $11 \equiv 3 \pmod{11}$
and that the matrix $S_1$ satisfies $\text{det}S_1=11^2$ with level 11.
Therefore, we can apply Theorem \ref{Theorem2.1} to $\vartheta_{S_1}^{(3)}$.
Namely, there is a modular form $G_1\in M_{12}(\Gamma_3)_{\mathbb{Z}_{(11)}}$ such that
$$
\vartheta_{S_1}^{(3)} \equiv G_1 \pmod{11}.
$$
We obtain the following table (for the abbreviation $T=[a,b,c]$, see \cite[(10)]{N-T}:
%%%%%%%Table1
\begin{table}[h]
\caption{Fourier coefficients: Degree 2 case}
\begin{center}
\begin{tabular}{c||rr}  \hline
$T$       &  $a(\vartheta_{S_1}^{(2)};T)$     &   $a(\vartheta_\alpha^{(2)};T)$ \\  \hline
$[0,0,0]$   &           1                                    &              1                           \\
$[1,0,0]$   &           4                                    &            1104                        \\
$[1,1,1]$   &          0                                     &            97152                      \\
$[1,0,1]$   &           8                                    &           1022304                    \\
\end{tabular}
\end{center}
\end{table}
%%%%%%%%%
\\
If we use Theorem \ref{Theorem2.2}, we obtain 
$\vartheta_\alpha^{(2)} \equiv G_1 \equiv \vartheta_{S_1}^{(2)} \pmod{11}$. This implies
that
$$
a(\vartheta_\alpha^{(3)};T) \equiv a(\vartheta_{S_1}^{(3)};T) \pmod{11}
$$
holds for all $T\in Sym_3^*(\mathbb{Z})$ with $\text{rank}T<3$. Moreover, we
obtain the following numerical data (for the abbreviation 
$T=[a,b,c,;d,e,f]$, see \cite[(7)]{N-T}):
\newpage
%%%%%%%Table2
\begin{table}[h]
\caption{Fourier coefficients: Degree 3 case}
\begin{center}
\begin{tabular}{c||rr}  \hline
$T$       &  $a(\vartheta_{S_1}^{(3)};T)$     &   $a(\vartheta_\alpha^{(3)};T)$ \\  \hline
$[1,1,1;1,1,1]$   &           0                                  &       4177536                         \\
$[1,1,1;0,0,1]$   &           0                                    &     81607680                    \\
$[1,1,1;0,0,0]$   &           0                                    &     781393536                     
\end{tabular}
\end{center}
\end{table}
%%%%%%%%%
By using Theorem \ref{Theorem2.2} again, we obtain
$$
a(\vartheta_\alpha^{(3)};T) \equiv a(\vartheta_{S_1}^{(3)};T) \pmod{11}
$$
for all $0\leq T\in Sym_3^*(\mathbb{Z})$. This completes the proof of (i).

The proofs of (ii) and (iii) are similar to that of (i). Concerning the
Coxeter numbers, we have the following numerical data:
%%%%%%%%%%%%%
\begin{table}[h]
\caption{Coxeter numbers}
\begin{center}
\begin{tabular}{c||rrr | rr}  
lattice                  &  $\delta$ & $\iota$ & $\chi$ & $\epsilon$ & $\omega$  \\  \hline
Coxeter number   &   25      &   14   &  3    &      22     &   0          \\         
\end{tabular}
\end{center}
\end{table}
%%%%%%%%%
\vspace{5mm}
\\
These imply that
$$
\vartheta_\delta^{(3)} \equiv \vartheta_\iota^{(3)} \equiv \vartheta_\chi^{(3)} 
\pmod{11},\qquad
\vartheta_\epsilon^{(3)} \equiv \vartheta_\omega^{(3)}  \pmod{11}.
$$
Thus, the proof is reduced to showing that
$$
\vartheta_\delta^{(3)} \equiv \vartheta_{S_2}^{(3)} \pmod{11},\qquad
\vartheta_\omega^{(3)} \equiv \vartheta_{S_3}^{(3)}  \pmod{11},
$$
which can be obtained from the following tables:
%%%%%%%%%
\begin{table}[h]
\caption{Fourier coefficients: Degree 2 case (Continued)}
\begin{center}
\begin{tabular}{c||rr | rr}  \hline
$T$       &  $a(\vartheta_{S_2}^{(2)};T)$     &   $a(\vartheta_\delta^{(2)};T)$ 
             &  $a(\vartheta_{S_3}^{(2)};T)$     &   $a(\vartheta_\omega^{(2)};T)$ 
\\  \hline
$[0,0,0]$   &           1        &           1      &    1    &           1          \\
$[1,0,0]$   &           6        &       600      &    0     &          0          \\
$[1,1,1]$   &          12      &       27600    &   0     &          0           \\
$[1,0,1]$   &           0       &     303600    &    0    &          0              
\end{tabular}
\end{center}
\end{table}
%%%%%%%%%
\newpage
%%%%%%%Table5
\begin{table}[h]
\caption{Fourier coefficients: Degree 3 case (Continued)}
\begin{center}
\begin{tabular}{c||rr | rr}  \hline
$T$       &  $a(\vartheta_{S_2}^{(3)};T)$     &   $a(\vartheta_\delta^{(3)};T)$ 
            &  $a(\vartheta_{S_3}^{(3)};T)$     &   $a(\vartheta_\omega^{(3)};T)$ 
\\  \hline
$[1,1,1;1,1,1]$   &           0      &    607200     &     0   &   0  \\
$[1,1,1;0,0,1]$   &           0     & 12751200     &    0    &   0   \\
$[1,1,1;0,0,0]$   &           0     &   127512000 &    0    &   0                                
\end{tabular}
\end{center}
\end{table}
%%%%%%%%%
%\newpage
This completes the proof of (ii) and (iii).
\end{proof}
%%%%%%%%%%%%%%%%%%%%%%%%%%%%%%%%%%%%%%%%%%%%%%%%%%%%%%%%%%%%%%
\section{Theta series for Leech lattice}
\label{sec:4}
In this section, we will show that the congruence relations described in the
previous section are extended to the case of the Leech lattice.

In \cite{O}, Ozeki provided a large number of data for the Fourier coefficients of theta series
for the Leech lattice. His numerical data enable us to extend the congruence
relation to the case of Leech lattice.
%%%%%%%
\begin{Thm}
\label{Theorem4.1}{\it
Let $\omega$ be the Leech lattice and $S_3$ the matrix given in
(\ref{Si}). Then, the following congruence relation holds:
$$
\vartheta_\omega^{(4)} \equiv \vartheta_{S_3}^{(4)} \pmod{11}.
$$
In particular,
$$
\varTheta (\vartheta_\omega^{(4)}) \equiv 0 \pmod{11},
$$
where $\varTheta$ is the theta operator defined in $\S$ \ref{sec:2-1}.}
\end{Thm}
%%%%%%%%%%%%%%%%%
%%%%%%%%%%%%%
\begin{proof}
We consider a modular form $G_3\in M_{12}(\Gamma_4)_{\mathbb{Z}_{(11)}}$ such that
$$
\vartheta_{S_3}^{(4)} \equiv G_3 \pmod{11}
$$
as in the case of degree 3.
The congruence relation that we should prove is
\begin{equation}
\label{degree4-1}
a(\vartheta_\omega^{(4)};T) \equiv a(\vartheta_{S_3};T)
(\equiv a(G_3;T)) \pmod{11}
\end{equation}
for all $0\leq T\in Sym_4^*(\mathbb{Z})$. As
$\vartheta_\omega^{(3)}\equiv \vartheta_{S_3}^{(3)} \pmod{11}$, it suffices
to prove (\ref{degree4-1}) for $T$ with $T>0$. Moreover, by considering
the result for the  Sturm bound (Theorem \ref{Theorem2.2}), it follows that
(\ref{degree4-1}) should be shown for $T$ satisfying
\begin{equation}
\label{degree4-2}
0<T=(t_{ij}) \in Sym_4^*(\mathbb{Z})\;\;\text{with}\;\;
t_{ii}\leq 2\quad (i=1,2,3,4)
\end{equation}
If any of $t_{11},\ldots,t_{44}$ is 1 under the condition (\ref{degree4-2})
on $T$, we have $a(\vartheta_\omega^{(4)};T)=0$ because the lattice $\omega$
has no vectors of length 2. Moreover, from the identity
$$
a(\vartheta_{S_3}^{(4)};T)=\sharp\{\,G\,\mid\, S_3[G]=2T\,\}
$$
and $\text{det}S_3=11^2$, we see that
$$
a(\vartheta_{S_3}^{(4)};T)=0
$$
for $T$ whose discriminant satisfies $d_T<11^2$. Consequently, it suffices to
prove (\ref{degree4-1}) for $0<T=(t_{ij})\in Sym_4^*(\mathbb{Z})$ with
$$
t_{11}=t_{22}=t_{33}=t_{44}=2.
$$
If we consider Table 6 in $\S$ 5, we see that
the target congruence relation (\ref{degree4-1})
should be proved only for $T$ with $d_T=11^2$. In the case $d_T=11^2$,
we have $a(\vartheta_{S_3}^{(4)};T)=|\text{Aut}(S_3)|=24$ (cf. \cite{N}).
Consequently, we obtain
$$
a(\vartheta_\omega^{(4)};T) =12599323656192000
                                          \equiv 24=a(\vartheta_{S_4}^{(4)};T) \pmod{11}
$$
for $T$ with $d_T=11^2$. This completes the proof of (\ref{degree4-1}).
The congruence relation $\varTheta (\vartheta_\omega^{(4)}) \equiv 0 \pmod{11}$
is a consequence of $\vartheta_\omega^{(4)} \equiv \vartheta_{S_3}^{(4)} \pmod{11}$.
\end{proof}
%%%%%%%%%%%%%%%%%%%%%%%%%%%%%%%%%%%%%%%%%%%%%%%%%%%%%%%%%%%%%%%%%%%
%\newpage
\section{Fourier coefficients of theta series for Leech lattice}
\subsection{Ozeki's calculation}
\label{sec5-1}
In this subsection, we present a table concerning the Fourier coefficients
of theta series for the Leech lattice calculated by Ozeki.
%%%%%%%%%
%\newpage
\begin{table}[h]
\caption{Fourier coefficients of degree 4 theta series for Leech lattice by Ozeki
\cite[Table 5]{O} : factorized.}
\label{Ozeki}
\begin{center}
\begin{tabular}{r l l}  \hline
$d_T$  &  $T$     &   $a(\vartheta_\omega^{(4)};T)$ \\ \hline
$64$  & (2,2,2,2,0,0,0,2,2,2)  & $2^8\cdot 3^8\cdot 5^3\cdot 7^2\cdot 11\cdot 13\cdot 23$  \\
$80$  & (2,2,2,2,2,0,0,2,0,2)  & $2^{11}\cdot 3^8\cdot 5^3\cdot 7^2\cdot 11\cdot 13\cdot 23$ \\
$81$  & (2,2,2,2,1,1,1,1,2,2)  & $2^{19}\cdot 3^3\cdot 5^3\cdot 7^2\cdot 11\cdot 13\cdot 23$ \\
$84$  & (2,2,2,2,1,0,0,2,2,2)  & $2^{16}\cdot 3^7\cdot 5^3\cdot 7\cdot 11\cdot 13\cdot 23$  \\
$96$  & (2,2,2,2,2,1,-1,0,0,2) & $2^{15}\cdot 3^7\cdot 5^3\cdot 7^2\cdot 11\cdot 13\cdot 23$  \\
$105$ & (2,2,2,2,2,1,0,0,1,2) & $2^{19}\cdot 3^7\cdot 5^3\cdot 7\cdot 11\cdot 13\cdot 23$  \\
$108$ & (2,2,2,2,2,1,-1,-1,1,-1) & $2^{19}\cdot 3^4\cdot 5^4\cdot 7^2\cdot 11\cdot 13\cdot 23$ \\
$112$ & (2,2,2,2,2,1,0,2,0,0) & $2^{16}\cdot 3^8\cdot 5^4\cdot 7\cdot 11\cdot 13\cdot 23$  \\
$116$ & (2,2,2,2,2,1,0,0,2,0) & $2^{16}\cdot 3^8\cdot 5^3\cdot 7^2\cdot 11\cdot 13\cdot 23$  \\
$120$ & (2,2,2,2,1,1,1,2,2,0) & $2^{18}\cdot 3^7\cdot 5^3\cdot 7^2\cdot 11\cdot 13\cdot 23$  \\
$121$ & (2,2,2,2,2,1,0,1,1,2) & $2^{20}\cdot 3^8\cdot 5^3\cdot 7^2\cdot 13\cdot 23$ 
\end{tabular}
\end{center}
\end{table}
\\
\\
\begin{table}[h]
\caption{
(Continued)}
\begin{center}
\begin{tabular}{r l l}  \hline
$d_T$  &  $T$     &   $a(\vartheta_\omega^{(4)};T)$ \\ \hline
$125$ & (2,2,2,2,1,1,-1,-1,1,1) & $2^{20}\cdot 3^9\cdot 5\cdot 7^2\cdot 11\cdot 13\cdot 23$  \\
$128$ & (2,2,2,2,0,0,0,2,2,0) & $2^{10}\cdot 3^9\cdot 5^4\cdot 7^2\cdot 11^2\cdot 13\cdot 23$  \\
$129$ & (2,2,2,2,1,1,1,1,2,2) & $2^{19}\cdot 3^7\cdot 5^3\cdot 7^2\cdot 11\cdot 13\cdot 23$  \\
$132$ & (2,2,2,2,2,1,-1,0,0,1) & $2^{17}\cdot 3^7\cdot 5^4\cdot 7^2\cdot 11\cdot 13\cdot 23$  \\
$140$ & (2,2,2,2,1,1,-1,0,0,2) & $2^{19}\cdot 3^8\cdot 5^4\cdot 7\cdot 11\cdot 13\cdot 23$  \\
$144$ & (2,2,2,2,2,1,-1,0,0,0) & $2^{16}\cdot 3^7\cdot 5^3\cdot 7^2\cdot 11\cdot 13\cdot 23^2$  \\
$144$ & (2,2,2,2,2,0,0,0,0,2) & $2^{12}\cdot 3^7\cdot 5^3\cdot 7^2\cdot 11\cdot 13\cdot 23\cdot 373$  \\
$145$ & (2,2,2,2,2,1,0,-1,-1,1) & $2^{19}\cdot 3^8\cdot 5^3\cdot 7^2\cdot 11\cdot 13\cdot 23$  \\
$153$ & (2,2,2,2,1,1,0,1,1,2) & $2^{19}\cdot 3^7\cdot 5^4\cdot 7^2\cdot 11\cdot 13\cdot 23$  \\
$156$ & (2,2,2,2,1,1,1,2,0,0) & $2^{20}\cdot 3^8\cdot 5^3\cdot 7^2\cdot 11\cdot 13\cdot 23$  \\
$160$ & (2,2,2,2,1,1,-1,1,-1,0) & $2^{15}\cdot 3^8\cdot 5^3\cdot 7^2\cdot 11\cdot 13\cdot 23\cdot 41$  \\
$160$ & (2,2,2,2,1,1,0,2,0,0) & $2^{18}\cdot 3^8\cdot 5^4\cdot 7^2\cdot 11\cdot 13\cdot 23$    
\end{tabular}
\end{center}
\end{table}
%%%%%%%%%
\\
\\
Here, the following abbreviation was used:
$$
T=(t_{ij})=:(t_{11},t_{22},t_{33},t_{44},t_{12},t_{13},t_{23},t_{14},t_{24},t_{34}).
$$
\subsection{Observation}
\label{sec5-2}
Herein, we provide
some congruences that can be obtained by calculations analogous to those by Ozeki.

Ozeki \cite{O} calculated the value $a(\vartheta_\omega^{(n)};T)$ for various
degrees $n$ \cite[Table 4,5,6]{O}.
We list the congruence relations expected from his tables.
\begin{align*}
& \bullet\quad \vartheta_\omega^{(n)} \equiv 1 \pmod{7}\\
& \bullet\quad \varTheta (\vartheta_\omega^{(4)}) \equiv \varTheta (\vartheta_\omega^{(5)})
    \equiv 0 \pmod{7^2}.
\end{align*}

%%%%%%%%%%%%%%%%%%%%%%%
%\begin{acknowledgements}
%This work was supported by JSPS KAKENHI. The first author: Grants-in-Aid (C)
%(No.25400031). The second author: Grants-in-Aid (S) (No. 23224001).
%\end{acknowledgements}

%%%%%%%%
Shoyu Nagaoka
\\
Department of Mathematics
\\
Kindai University
\\
Higashi-Osaka
\\
Osaka 577-8502
\\
Japan

\end{document}